\documentclass[reqno]{amsart}
\usepackage{amssymb}
\usepackage{graphicx}
\usepackage{amsthm}
\usepackage{amsfonts}
\usepackage{latexsym}
\usepackage{amsmath}
\usepackage{setspace}
\usepackage{xcolor}

\newtheorem{theorem}{Theorem}[section]

\newtheorem{proposition}[theorem]{Proposition}
\newtheorem{corollary}[theorem]{Corollary}
\newtheorem{lemma}[theorem]{Lemma}

\theoremstyle{definition}
\newtheorem{definition}[theorem]{Definition}
\newtheorem{remark}[theorem]{Remark}
\newtheorem{example}[theorem]{Example}

\title{Generalized evolution semigroups and general dichotomies}

\author[N. Lupa]{Nicolae Lupa}
\address{N. Lupa, Department of Mathematics, Politehnica University of Timi\c soara,
Pia\c ta Victoriei 2, 300006 Timi\c soara, Romania}
\email{nicolae.lupa@upt.ro}

\author[L. H. Popescu]{Liviu Horia Popescu}
\address{L. H. Popescu, Department of Mathematics and Informatics, Faculty of Sciences,
University of Oradea, Universit\u a\c tii St. 1, 410087 Oradea, Romania}
\email{lpopescu2002@yahoo.com}

\begin{document}

\begin{abstract}
We introduce a special class of real semiflows, which is used to define a general type of evolution semigroups, associated to not necessarily exponentially bounded evolution families. Giving spectral characterizations of the corresponding generators, our results directly apply to a wide class of dichotomies, such as those with time-varying rate of change.
\end{abstract}

\maketitle

\section{Introduction}

It is known that if $A$ is a \emph{bounded} linear operator on a Banach space $X$, then the asymptotic behavior of solutions $x(t)=e^{tA}x_0$ of the autonomous equation
\begin{equation}\label{eq.aut}
\dot{x}=Ax
\end{equation}
is completely determined by the spectrum of $A$, $\sigma(A)$ (for much more information we refer the reader to  \cite[Chapter II]{Da} and \cite[Section I.3]{Eng}).
For instance, the differential equation \eqref{eq.aut} is  exponentially stable (i.e., there exist constants $\nu>0$ and $N\geq 1$ such that $\| e^{tA} \|\leq Ne^{-\nu t}$ for every $t\geq 0$) if and only if the spectrum $\sigma(e^{tA})$ lies in the open unit disk $\mathbb{D}=\{\lambda\in\mathbb{C}:\,|\lambda|<1\}$ or, equivalently,  $\sigma(A)$ is contained in the open left half of the complex plane
(see \cite[Theorem I.3.14]{Eng}). Consequently, the following statements are equivalent:
\begin{itemize}
\item the differential equation (1) admits an exponential dichotomy on $X$, i.e., there exists a direct decomposition $X=X_s\oplus X_u$, into $A$-invariant closed subspaces $X_s$  and  $X_u$, such that there are constants $\nu>0$, $N\geq 1$ satisfying
$$ \| e^{tA}x \| \leq N e^{-\nu t} \| x\|, \; x\in X_s,$$
and
$$ \| e^{tA}x \| \geq \frac{1}{N} e^{\,\,\nu t} \| x\|, \; x\in X_u,$$
for all $t\geq 0$;
\item the semigroup $\left\{ e^{tA} \right\}_{t\geq 0}$ is hyperbolic, i.e., $\sigma(e^{tA})\cap \mathbb{T}=\emptyset$
for some (and hence all)
$t>0$, where $\mathbb{T}$ denotes the unit circle on $\mathbb{C}$;
\item the spectrum $\sigma(A)$ does not intersect the imaginary axis, $\sigma(A)\cap i\mathbb{R}=\emptyset$.
\end{itemize}

The key to the proof of the above results lies in the spectral mapping property
$$\sigma(e^{tA})=e^{t\sigma(A)},  \; t\geq 0,$$
which in this case is a simple consequence of Dunford's  functional calculus \cite[Theorem 2.9]{Du}.

For a general $C_0$-semigroup $\left\{ S_t \right\}_{t\geq 0}$ on $X$, generated by an \emph{unbounded} linear operator $A$, the situation is more complex.
We recall that a  \emph{$C_{0}$-semigroup} on the Banach space $X$ is a family $\left\{ S_t\right\}_{t\geq 0}$ of bounded linear operators on $X$ such that the following properties are satisfied:
\begin{itemize}
\item $S_0=\mathrm{Id}$ (the identity on $X$);
\item $S_t \,S_\tau=S_{t+\tau}, \text{ for }t,\tau\geq 0;$
\item $S_t x\to x$ in $X$ as $t\to 0^+$ for all $x\in X$.
\end{itemize}
The  \emph{generator} of a $C_{0}$-semigroup $\left\{ S_t\right\}_{t\geq 0}$ is the linear operator $A$, with domain $D(A)$, defined by
$$D(A)=\left\{ x\in X:\, \lim\limits_{t\to 0^+} \frac{S(t)x-x}{t} \text{ exists }\right\}$$
and
$$Ax= \lim\limits_{t\to 0^+} \frac{S(t)x-x}{t},\, x\in D(A).$$
We refer the reader to the monograph of Engel and Nagel \cite{Eng} for a brief history of $C_0$-semigroups (see also \cite{Na.Rha} for a short and very nice  presentation of semigroups and their recent applications).

For any $C_0$-semigroup, a spectral inclusion is always valid (see, for instance, Theorem 2.3 in \cite[Chapter 2]{Paz}):
$$\sigma(S_t)\supseteq e^{t\sigma(A)}, \, t\geq 0.$$
On the other hand, if we take a $C_0$-semigroup that cannot be extended to a group (e.g., the left translation semigroup on the Banach space of all continuous functions on $\mathbb{R}_+$ vanishing at infinity), then $0\in \sigma(S_t)$ for all $t>0$, while evidently $0$ does not belong to $e^{t\sigma(A)}$. Hence, to obtain a spectral mapping property,
we must exclude $0$ from the spectrum of $S_t$.
Even so, the \emph{spectral mapping theorem}
\begin{equation*}\label{eq.smt}
\sigma(S_t)\setminus\{0\}=e^{t\sigma(A)}, \, t\geq 0,
\end{equation*}
generally fails (see  \cite[Section 2.1.5]{Ch.La.1999} or \cite[pp. 270-275]{Eng}). However, it is valid for many classes of $C_0$-semigroups, such as eventually norm continuous semigroups  (see \cite[Paragraph IV.3.10]{Eng} or \cite[Theorem 2.3.2]{Nee.B}) and, consequently (see Diagram (4.26) in \cite[p. 119]{Eng}), eventually compact semigroups, eventually differentiable semigroups, analytic semigroups, uniformly continuous semigroups \cite[Corollary IV.3.12]{Eng}. Furthermore, some spectral mapping formula
holds for the point and residual spectrum  (see  \cite[Theorems 2.1.2,  2.1.3]{Nee.B}).
Thus, the failure of the spectral mapping theorem is determined by the approximate point spectrum.
For a  detailed analysis of the spectral mapping theorem we refer the reader to \cite[Section IV.3]{Eng} or \cite[Chapter 2]{Nee.B}.

In order to overcome the failure of the spectral mapping theorem,  Latushkin and Montgomery-Smith \cite{Lat1-1,Lat1} proved that
the hyperbolicity of  $\left\{ S_t \right\}_{t\geq 0}$  can be completely determined by the generator $G$ of a $C_0$-semigroup $\left\{ T_t \right\}_{t\geq 0}$, defined by
\begin{equation*}\label{es.aut}
(T_tu)(s)=S_t u(t-s),
\end{equation*}
on certain \emph{super-space} $E(X)$ of functions $u:\mathbb{R}\to X$. This semigroup is called the \emph{evolution semigroup} associated to $\left\{ S_t \right\}_{t\geq 0}$.
For instance, if  $E(X)$ is one of the spaces  $C_0(\mathbb{R},X)$ or $L^p(\mathbb{R},X)$, for $1\leq p<\infty$, then the hyperbolicity of $\left\{ S_t \right\}_{t\geq 0}$ on $X$ is equivalent to the  hyperbolicity of $\left\{ T_t \right\}_{t\geq 0}$ on $E(X)$ or, equivalently, to the invertibility of $G$ (see \cite[Theorem 2.5]{Lat1} or \cite[Theorem 2.39]{Ch.La.1999}).
Moreover, the evolution semigroup $\left\{ T_t \right\}_{t\geq 0}$ always satisfies the spectral mapping theorem in $E(X)$ even if the underlying semigroup $\left\{ S_t \right\}_{t\geq 0}$ does not have this property on $X$ \cite[Corollary 2.40]{Ch.La.1999}.

For a non-autonomous differential equation $\dot{x}=A(t)x$ the situation is much more difficult then the autonomous case.
It is known that if the associated  Cauchy problem
\begin{equation*}\label{eq.Cauchy}
\quad\left\{\begin{array}{ll}
\dot{x}(t)=A(t)x(t),\;t\geq{s},\\
x(s)=x_s,
\end{array}\right.
\end{equation*}
is well-posed, then one can define an evolution family $\left\{U(t,s)\right\}_{t\geq s}$ on $X$ such that $x(t)=U(t,s)x(s)$ for $t\geq s$ (for more details on well-posed non-autonomous Cauchy problems we refer the reader to \cite{Na.Ni} and the references therein).

Let now $\mathcal{U}=\left\{U(t,s)\right\}_{t\geq s}$ be an {evolution family} on $X$ (not necessarily associated to a non-autonomous differential equation) such that there exist constants $\alpha>0$ and $K\geq 1$ satisfying
\begin{equation}\label{eq.ueg}
\| U(t,s)\|\leq K e^{\alpha(t-s)}, \text{ for } t\geq s,
\end{equation}
that is $\mathcal{U}$ is  (uniformly) \emph{exponentially bounded}. Then, one can define a $C_0$-semigroup $\left\{  T_{t} \right\}_{t\geq 0}$ on $C_0(\mathbb{R},X)$ or $L^p(\mathbb{R},X)$, for $1\leq p<\infty$, by
\begin{equation}
T_{t}u(s)=U(s,s-t)u(s-t), \label{eq.ces}%
\end{equation}
called the \emph{evolution semigroup} associated to the evolution family $\mathcal{U}$ (see \cite[Lemma VI.9.10]{Eng} or \cite[Proposition 3.11]{Ch.La.1999}), and denote its generator by $(G,D(G))$.

The theory of evolution (semi)groups associated to an evolution family has a long history going back to Howland \cite{How} and Lovelady \cite{Lov}.
Significant contributions to this theory are due to Evans \cite{Ev}, Neidhardt \cite{Nei},
Rau \cite{Rau1,Rau2}, Van Minh \cite{Min2,Min3}, Latushkin and Montgomery-Smith \cite{Lat1-1,Lat1}, Latushkin and Randolph \cite{Lat3},  R\"{a}biger and Schnaubelt \cite{Rab}, R\"{a}biger,  Rhandi, and Schnaubelt \cite{RRS}, and many others.

The spectra of the operators $T_t$ and the generator $G$ of the  evolution semigroup $\left\{  T_{t} \right\}_{t\geq 0}$ have some important symmetry properties:
\emph{$\sigma\left(  T_{t} \right)$ is rotationally invariant for each  $t>0$}, that is
$$
\lambda\sigma\left(  T_{t} \right)=\sigma\left(  T_{t} \right) \text{ for every $\lambda\in\mathbb{C}$ with $|\lambda | =1$},
$$
and \emph{the spectrum $\sigma(G)$ is invariant under translations along the imaginary axis}, respectively, i.e.,  $$\sigma(G)=\sigma(G)+i\mathbb{R}.$$
Furthermore, \emph{the evolution semigroup is hyperbolic if and only if its generator is invertible} (see \cite[Theorem 3.1]{Lat1} or \cite[Theorem 3.13]{Ch.La.1999}). In particular, \emph{the evolution semigroup satisfies the spectral mapping theorem}.

On the other hand, the hyperbolicity of the evolution semigroup characterizes the exponential dichotomy of the underlying evolution family (see Theorem \ref{th.ued} below),  and thus  the evolution semigroups method provides a strong tool to study the exponential dichotomy of evolution families. More precisely, the evolution family $\mathcal{U}$  admits a (uniform) \emph{exponential dichotomy} if:
\begin{enumerate}
\item[(a)] there exist projections $P(t):X\to X$, $t\in \mathbb{R}$, and write $Q(t)=\mathrm{Id}-P(t)$, \emph{compatible with $\mathcal{U}$}, that is
$P(t)U(t,s)=U(t,s)P(s)$
and the restriction $U_Q(t,s):Q(s)X\to Q(t)X$ of $U(t,s)$ is invertible, for all $t\geq s$;
\item[(b)] there exist constants $\nu>0$, $N\geq 1$ such that
$$\| U(t,s)P(s)\| \leq Ne^{-\nu (t-s) }  \text{ and }  \| U_Q(s,t)Q(t)\| \leq Ne^{-\nu (t-s) }, \; t\geq s,$$
where $U_Q(s,t)=U_Q(t,s)^{-1}$.
\end{enumerate}

The above mentioned result is now stated in the following theorem (see, for instance, \cite[Theorem 3.17 and Theorem 4.25]{Ch.La.1999} or \cite[Theorem VI.9.18]{Eng}).

\begin{theorem}\label{th.ued}
Let $\mathcal{U}=\left\{U(t,s)\right\}_{t\geq s}$ be an exponentially bounded evolution family on a Banach space $X$, let $\left\{  T_{t} \right\}_{t\geq 0}$ be the associated evolution semigroup on $E(X)$ defined by \eqref{eq.ces}, where $E(X)$ is $C_0(\mathbb{R},X)$ or $L^p(\mathbb{R},X)$, for  $1\leq p<\infty$, and denote $G$  its generator. The following assertions are equivalent:
\begin{enumerate}
\item $\mathcal{U}$ admits an exponential dichotomy on $X$;
\item $\left\{  T_{t} \right\}_{t\geq 0}$ is hyperbolic on $E(X)$;
\item $\sigma(G) \cap i\mathbb{R}=\emptyset$.
In this case, $G$ is invertible and its inverse is given by
\begin{equation}\label{eq.green1}
(G^{-1}f)(t)=-\int_\mathbb{R} \Gamma(t,s)f(s)\, ds, \text{ for all } f\in E(X) \text{ and } t\in\mathbb{R},
\end{equation}
where
$$
\Gamma(t,s)=\begin{cases}
\;\;\phantom{-}U(t,s)P(s), & t>s,\\
-U_Q(t,s)Q(s), & t<s,
\end{cases}
$$
is the Green function associated to the exponential dichotomous evolution family $\mathcal{U}$.
\end{enumerate}
\end{theorem}

To our knowledge, formula \eqref{eq.green1} was first proved in this context by  Latushkin and Randolph \cite{Lat3}.

R\"{a}biger and Schnaubelt extended Theorem \ref{th.ued} to a large class of $X$-valued function spaces, including in particular $C_0(\mathbb{R},X)$ and $L^p(\mathbb{R},X)$, for $1\leq p <\infty$  \cite{Rab}. For an excellent survey on the theory of evolution semigroups and their applications we refer the reader  to the monograph of Chicone and Latushkin \cite{Ch.La.1999}, where the authors also developed a systematic theory of evolution semigroups induced  by a cocycle
over a locally compact metric space acting on Banach fibers.
More recently, the evolution semigroups method was extended to nonuniform behavior (see, for instance, \cite{BPV0,BPV1,BV1,Lu.Po.JDEA,Lu.Po.Sem}).

For non-autonomous differential equations, the classical concept of exponential dichotomy may look as too restrictive, therefore it is important to search for more general behavior, for instance considering dichotomies given by general growth rates, an approach initiated to the best of our  knowledge by Naulin and Pinto \cite{Na.Pi.94, Na.Pi}. In particular, the contraction and expansion can be of the form $e^{\gamma\int_0^t \rho(\xi) d\xi}$ (see \cite{Jiang,Jiang2,Martin,Mul}) or  $e^{\gamma\mu(t)}$ \cite{Ba.Va.1}, including the usual exponential behavior and polynomial behavior \cite{Ba.Va.4,Be.Si} as special cases. These types of asymptotic behavior can occur in the critical situations when all Lyapunov exponents are infinite ($\pm\infty$) or they are all  zero \cite{Ba.Va.1}. We also emphasize that some of these types of dichotomy were connected to the important problem of linearization of dynamical systems, which is of a great interest for geometry.

In this paper, we generalize the classical concept of the evolution semigroup associated to an evolution family, hereby called the \emph{generalized evolution semigroup}, replacing the right translation semiflow $\varphi(t,s)=s-t, \;t\geq 0, \, s\in\mathbb{R},$ in formula \eqref{eq.ces} with a general real semiflow.
We  point out  that in our setting the evolution family might not be exponentially bounded.

The main purpose of generalizing the concept of evolution semigroup is to study a wide class of dichotomies presented above.
We prove that under some suitable conditions, the evolution semigroup we consider is similar to the classical evolution semigroup defined by \eqref{eq.ces}. This enables us to apply well established results to completely characterize a general type of exponential dichotomy of an evolution family in terms of spectral conditions imposed to the generator of the corresponding generalized evolution semigroup.
We restrict our study to the case of the space of all continuous functions vanishing at infinity.

Present paper is organized as follows.
In Section \ref{sec.2} we introduce the notion of a \emph{real semiflow}. We think
the main result here is the one in Theorem \ref{th4}, that completely characterizes the non-degenerated real semiflows.
The third section deals with what we call as \emph{generalized evolution semigroups}, associated to real semiflows.
In Theorem \ref{th1} we prove that the generalized evolution semigroups are $C_{0}$-semigroups on $C_0(\mathbb{R},X)$. Furthermore, in subsection \ref{p.similar} we show that the generalized evolution semigroups of a certain type are similar to the classical evolution semigroup, thus satisfying the spectral mapping theorem.
In the last section  we analyze a wide class of exponential dichotomies with
time-varying rates of change, by using the generalized evolution semigroups previously constructed.

\section{Real semiflows}\label{sec.2}

In this section, we introduce a special class of real semiflows which will be used to generalize the classical notion of the evolution semigroup associated to an evolution family. To the best of our knowledge, the results in this section are completely new.

As usually, $\mathbb{R}$ denotes the field of real numbers and write $\mathbb{R}_+=[0,\infty)$.
\begin{definition}\label{def4}\rm
Let $\Theta$ be a locally compact metric space.
A continuous mapping $\varphi:\mathbb{R}_+\times \Theta\to \Theta$  is called a
\emph{semiflow} on $\Theta$ if the following properties hold:
\begin{enumerate}
\item[(i)]   $\varphi_0(\theta)=\theta$, for every $\theta\in\Theta$,
\item[(ii)]  $\varphi_{t}\circ\varphi_{\tau}=\varphi_{t+\tau}$, for all $t,\tau\geq 0$,
\end{enumerate}
where $\varphi_t(\theta)=\varphi(t,\theta)$, for $t\geq 0$ and $\theta\in\Theta$. In this case, the function $\varphi_t(\cdot)$ is  called \emph{the transition map} defined by $t$.

In particular, a semiflow $\varphi:\mathbb{R}_+\times \mathbb{R}\to \mathbb{R}$ (that is $\Theta=\mathbb{R}$) will be called a \emph{real semiflow} if in addition the following inequality holds:
\begin{equation}\label{eq.ineq.def}
\varphi_{t}(s)\leq s, \text{ for all  $t\geq 0$ and $s\in\mathbb{R}$.}
\end{equation}
\end{definition}

It is easy to observe that \eqref{eq.ineq.def} implies
\begin{equation}\label{eq.-oo}
\lim\limits_{s\to -\infty}\varphi_t(s)=-\infty, \text{ for every } t\geq 0.
\end{equation}

The \textit{orbit} of  $s\in\mathbb{R}$ is the set
$o\left(  s\right)=\left\{  \varphi_{t}\left(  s\right):\;t\geq0\right\}.$
If $o\left(  s\right)  =\left\{  s\right\}$, we say that the orbit
$o\left(s\right)$ is \emph{trivial}.

A classical example of real semiflow is \emph{the right translation semiflow}, defined by
\begin{equation}\label{eq.translation}
\varphi(t,s)=s-t, \text{ for } t\geq 0 \text{ and } s\in\mathbb{R},
\end{equation}
with the orbits $o(s)=(-\infty,s]$, $s\in\mathbb{R}$.

\begin{remark}\label{remark1}
For each $s\in\mathbb{R}$ the function
$\mathbb{R}_+\ni t\mapsto\varphi_{t}\left(  s\right)\in\mathbb{R}$ is
decreasing, meanwhile for each $t\geq 0$ the function $\mathbb{R}\ni s\mapsto\varphi_{t}\left(  s\right)\in\mathbb{R}$ is increasing.
\end{remark}
\begin{proof}
Indeed, for each fixed $s\in\mathbb{R}$, if
$t_1,t_2\geq 0$ with  $t_{1}<t_{2}$, one has
$$\varphi_{t_{2}}\left(  s\right)=\varphi_{t_{2}-t_{1}}\left(  \varphi_{t_{1}}\left(  s\right)  \right)\leq\varphi_{t_{1}}\left(  s\right),$$
and thus the function $t\mapsto\varphi_{t}\left(  s\right)$ is decreasing. For the second statement, assume that there exist
$s_1,s_2\in\mathbb{R}$ with $s_{1}<s_{2}$  such that
$\varphi_{t_{0}}\left(  s_{1}\right) >\varphi_{t_{0}}\left(  s_{2}\right)$ for some $t_0>0$. Then,
$$\varphi_{t_0}(s_2)<\varphi_{t_{0}}\left(  s_{1}\right)  \leq s_{1}< s_{2}=\lim\limits_{t\to 0_+}\varphi_{t}\left(  s_{2}\right),$$  and hence there exists
$t_{1}\in(0,t_0)$ such that $\varphi_{t_{1}}\left(  s_{2}\right)  =s_{1}$. This yields
$$\varphi_{t_{0}}\left(  s_{2}\right)  <\varphi_{t_{0}}\left(  \varphi_{t_{1}}\left(  s_{2}\right)  \right)=\varphi_{t_{0}+t_{1}}\left(
s_{2}\right),$$ which contradicts the first statement above.
\end{proof}

\medskip

For  any fixed $s\in\mathbb{R}$, we  set $$\omega\left(  s\right)  =\lim\limits_{t\rightarrow\infty}\varphi_{t}\left(  s\right).$$

Notice that $\omega(s)$ is a real number or $\omega(s)=-\infty$, and
\begin{equation}\label{eq.w}
\omega(s)\leq \varphi_t(s)\leq s, \text{ for all }  t\geq 0 \text{ and } s\in\mathbb{R}.
\end{equation}
Moreover,
\begin{equation}\label{w-monoton}
s_1<s_2 \text{ in } \mathbb{R} \;\Rightarrow \; \omega(s_1)\leq \omega(s_2) \text{ in } \mathbb{R}\cup \{-\infty\}.
\end{equation}
Hence, if  $\omega(s)=-\infty$ for some $s\in \mathbb{R}$, then $\omega(\tau)=-\infty$ for every $\tau\in (-\infty,s]$.
On the other hand, if there exists $s\in\mathbb{R}$ such that $\omega(s)\in\mathbb{R}$, then $\omega(\tau)\in \mathbb{R}$ for every $\tau\geq s$, thus $\omega:[s,\infty)\to\mathbb{R}$ is well-defined. Furthermore,  \emph{$\omega$ is right continuous at $s$}.
Indeed, since  $t\mapsto\varphi_{t}(s)$ is a decreasing function, we have
$$\omega(s)=\inf\limits_{t\geq 0} \varphi_t(s).$$
It is well-known that the pointwise infimum of an arbitrary collection of upper semi-continuous functions is upper semi-continuous. Thus,
\begin{equation*}\label{upper-semicontinuous}
\limsup\limits_{\tau\to s^+} \,\omega(\tau)\leq \omega(s).
\end{equation*}
On the other hand, by \eqref{w-monoton} we have
$$\omega(s)\leq \lim\limits_{\tau\to s^+} \omega(\tau)=\limsup\limits_{\tau\to s^+}\, \omega(\tau).$$
Therefore,
$\lim\limits_{\tau\to s^+} \omega(\tau)=\omega(s),$
that is $\omega$ is right continuous at $s$.

\medskip

Now, let us observe that letting $\tau\to\infty$ in the relation
$$\varphi_\tau(\varphi_t(s))=\varphi_{\tau+t}(s)=\varphi_t(\varphi_\tau(s)),$$
we get
\begin{equation}\label{eq.prop.a}
\omega(\varphi_t(s))=\omega(s),\; t\geq 0,
\end{equation}
and, in addition, using the continuity of the function $\varphi_t(\cdot)$, we obtain
\begin{equation}\label{eq.prop.b}
\varphi_t(\omega(s))=\omega(s),\; t\geq 0.
\end{equation}
On the other hand, letting  $t\to\infty$ in the above relation, it results that
\begin{equation}\label{eq.prop.c}
\omega(\omega(s))=\omega(s).
\end{equation}
We stress that relations \eqref{eq.prop.a}--\eqref{eq.prop.c} are valid provided that $\omega(s)\in\mathbb{R}$.
In particular, \eqref{eq.prop.a} shows  that the restriction of $\omega$ on the orbit
$o(s)$ is constant for each fixed $s\in\mathbb{R}$ with $\omega(s)\in\mathbb{R}$. In fact, $\omega(\tau)=\omega(s)$ for every $\tau\in o(s)$.

\begin{definition}\label{ndg} \rm
A \emph{non-degenerate semiflow} is a real semiflow without trivial orbits and, if contrary, we call it \emph{degenerate}.
\end{definition}

For instance, the right translation semiflow, defined in \eqref{eq.translation}, is a non-degenerate  semiflow.

The next result gives some simple equivalent conditions for a real semiflow to be non-degenerate.

\begin{lemma}\label{lemma2}
Let $\varphi:\mathbb{R}_+\times \mathbb{R}\to \mathbb{R}$ be a real semiflow. The following statements are equivalent:
\begin{enumerate}
\item[(i)]    $\varphi$ is non-degenerate;
\item[(ii)]   $\lim\limits_{t\rightarrow\infty}\varphi_{t}(s)  =-\infty,$ for every $s\in\mathbb{R}$;
\item[(iii)]   $\varphi_{t}\left(  s\right)  <s$, for all  $t>0$ and $s\in\mathbb{R}$;
\item[(iv)]   For each $s\in\mathbb{R}$ there exists $t_{s}>0$ such that $\varphi_{t_{s}}(s) \neq s$.
\end{enumerate}
\end{lemma}
\begin{proof}
(i) $\Rightarrow$ (ii)  If there exists $s_0\in\mathbb{R}$ such that  $\lim\limits_{t\rightarrow\infty}\varphi_{t}(s_0)$ is finite, which is equivalent to $\omega(s_0)\in\mathbb{R},$ then \eqref{eq.prop.b} shows  that
the orbit $o\left(  \omega\left(s_0\right)  \right)$ is trivial and hence $\varphi$ is degenerate, which contradicts (i).

(ii) $\Rightarrow$ (iii)
Assume that there exists $t_{0}>0$
and $s_{0}\in\mathbb{R}$ such that  $\varphi_{t_{0}}\left( s_{0}\right)=s_{0}.$ Then,
$$\varphi_{nt_{0}}\left(  s_{0}\right)  =s_{0}, \text{ for all }n\in\mathbb{N},$$
and letting $n\to \infty$ we get a contradiction.
Implications (iii) $\Rightarrow$ (iv) and (iv) $\Rightarrow$ (i) are trivial.
\end{proof}

Using the equivalence (i) $\Leftrightarrow$ (iii) in the previous lemma and the same type of arguments as in Remark \ref{remark1}, one can easily prove that \emph{if $\varphi:\mathbb{R}_+\times \mathbb{R}\to \mathbb{R}$ is a non-degenerate semiflow, then the function  $t\mapsto\varphi_{t}\left(  s\right)$ is strictly decreasing for each $s\in\mathbb{R}$ and the function
$s\mapsto\varphi_{t}\left(  s\right)$ is strictly increasing for each $t\geq 0$}.

\medskip

The next result is of a significant importance.  It states that any
non-degenerate semiflow is generated by a continuous, strictly increasing real function.

\begin{theorem}\label{th4}
A mapping $\varphi:\mathbb{R}_+\times \mathbb{R}\to \mathbb{R}$ is a non-degenerate semiflow if and only if there exists a continuous, strictly increasing function $\mu:\mathbb{R}\to\mathbb{R}$
with  $\lim\limits_{s\rightarrow-\infty}\mu (s) =-\infty$ and
\begin{equation}
\varphi_{t}(s)  =\mu^{-1}(\mu(s)-t), \text{ for all } t\geq 0 \text{ and } s\in\mathbb{R}. \label{eq5}%
\end{equation}
\end{theorem}

\begin{proof}
Suppose that $\varphi$ is a non-degenerate semiflow. Let $s,\tau\in\mathbb{R}$ with $s>\tau$.  According to Lemma \ref{lemma2} we have
$$-\infty =\lim\limits_{t\rightarrow\infty}\varphi_{t}\left(  s\right)<\tau<s=\varphi_0 (s),$$
and thus there exists a unique number
$t(s,\tau)>0$ such that
$\varphi_{t(s,\tau)}(s)  =\tau$.  On the other hand, since $\varphi_0(s)=s$, we can consider $t(s,\tau)=0$ whenever  $s=\tau$. Hence, $t(s,\tau)\geq 0$ is the unique number which satisfies the identity
\begin{equation}\label{eq.phi}
\varphi_{t(s,\tau)}(s)  =\tau, \text{ for } s\geq \tau.
\end{equation}

Let us first prove that
\begin{equation*}\label{eq.t}
t(s,\tau)+t(\tau,\eta)=t(s,\eta), \text{ for } s\geq \tau\geq \eta.
\end{equation*}
Indeed, for $s\geq \tau\geq \eta$  we have
\begin{align*}
\varphi_{t(s,\tau)+t(\tau,\eta)}(s)=\varphi_{t(\tau,\eta)}\left(\varphi_{t(s,\tau)}(s)\right)=\varphi_{t(\tau,\eta)}(\tau)=\eta=\varphi_{t(s,\eta)}(s).
\end{align*}

Define  a
function  $\mu:\mathbb{R}\to\mathbb{R}$, by
\begin{equation*}
\mu(s)  =
\begin{cases}
\phantom{-}t(s,0),& s\geq 0,\\
-t(0,s),& s<0.
\end{cases}
\end{equation*}
One can easily check that
\begin{equation}\label{eq.t.mu}
t(s,\tau)=\mu(s)-\mu(\tau), \text{ for } s\geq \tau.
\end{equation}

Since $t(s,\tau)>0$ whenever $s>\tau$,  the above identity yields that $\mu$ is  strictly increasing.
We now prove the continuity of  $\mu$. Let $s\in\mathbb{R}$ and choose a strictly increasing sequence $(s_n)\subset\mathbb{R}$ with $s_n<s$ and $s_n\to s$.
Set $t_n=\mu(s_n)$ and $t=\mu(s)$.
Since $\mu$ is a strictly increasing function,  we get that $(t_n)$ is  strictly increasing
and thus there exists $t'=\lim\limits_{n\to\infty} t_n\leq t.$
If $s>0$, then $s_n>0$ for sufficiently large $n\in\mathbb{N}$, hence
$t_n=t(s_n,0)$ and, by \eqref{eq.phi}, we have
$$\varphi_{t_n}(s_n)=\varphi_{t(s_n,0)}(s_n)=0.$$
On the other hand, $\varphi_{t_n}(s_n)\to \varphi_{t'}(s)$ and thus $\varphi_{t'}(s)=0$.  Since $t=t(s,0)$ is the unique number satisfying $\varphi_{t}(s)=0$,  then $t'=t$.
If
$s\leq 0$, then $s_n<0$ for every $n\in\mathbb{N}$,  consequently $t_n=-t(0,s_n)$ for  $n\in\mathbb{N}$. Then, we have
$$\varphi_{-t_n}(0)=\varphi_{t(0,s_n)}(0)=s_n.$$
This yields $\varphi_{-t'}(0)=s$. We also have
$\varphi_{-t}(0)=\varphi_{t(0,s)}(0)=s,$
which implies that $t=t'$.
All above considerations shows that  $\mu$ is left continuous at $s$. Similarly, one can prove that $\mu$ is right continuous, therefore it is a continuous function.

Observe  now that \eqref{eq.phi} and \eqref{eq.t.mu} imply
\begin{equation*}\label{eq.ts}
t(s,\tau)=\mu(s)-\mu(\varphi_{t(s,\tau)}(s)), \text{ for } s\geq \tau.
\end{equation*}
This yields
\begin{equation}\label{eq.mu}
t=\mu(s)-\mu(\varphi_t(s)), \text{ for all }  t\geq 0 \text{ and } s\in\mathbb{R}.
\end{equation}

Since $\mu$ is strictly increasing, there exists $l=\lim\limits_{s\to -\infty} \mu(s)\in \mathbb{R}\cup \{-\infty\}.$ If we assume that $l$ is finite, then  letting $t\to\infty$ in the identity
$$\mu(s)=t+\mu(\varphi_t(s))$$
and using Lemma \ref{lemma2} and the continuity of $\mu$, we  get that $\mu(s)=\infty+l$, which is false, thus  $\lim\limits_{s\rightarrow-\infty}\mu (s) =-\infty$.
If we put
$$\ell=\lim\limits_{s\rightarrow\infty}\mu (s)\in\mathbb{R}\cup\{+\infty\},$$
then the function $\mu:\mathbb{R}\to (-\infty,\ell)$
is invertible, with continuous inverse, and by \eqref{eq.mu} we get  \eqref{eq5}.

Conversely, one can easily check that the identity \eqref{eq5} defines a non-degenerate semiflow.
\end{proof}

In the following,
$\mathcal{A}$ denotes the set of all continuous, strictly increasing functions $\mu:\mathbb{R}\to (-\infty,\ell)$ with
$$\lim\limits_{s\rightarrow-\infty}\mu (s) =-\infty \text{ and } \lim\limits_{s\rightarrow+\infty}\mu (s) =\ell\in\mathbb{R}\cup\{+\infty\}.$$

Theorem \ref{th4} shows that  $\varphi:\mathbb{R}_+\times \mathbb{R}\to \mathbb{R}$ is a non-degenerate semiflow if and only if there exits $\mu\in\mathcal{A}$ such that
\eqref{eq5} holds.
Furthermore,
the following alternative
holds true:
\emph{either all the maps $s\mapsto\varphi_{t}(s)$, $t>0$,
are bounded, if $\mu$ is bounded by above, i.e.,   $\ell\in\mathbb{R}$, or they are all
unbounded, if $\ell  =+\infty$.}

\begin{example}
Setting $\mu(s)=-e^{-s}$, for  $s\in\mathbb{R}$, we obtain the first situation above, with $\ell=0$. On the other hand, if $\mu(s)=\mathrm{sign}(s)\ln(1+|s|)$ or  $\mu(s)=s^{2n+1}$, for some fixed $n\in\mathbb{N}$, we get $\ell=+\infty$. In particular, for $n=0$, that is  $\mu(s)=s$, we obtain the right translation semiflow.
\end{example}

\section{Evolution semigroups}\label{sec.3}

In the following, $X=(X,\| \cdot \|)$ is a Banach space and $\mathcal{B}(X)$ denotes the Banach algebra of all bounded linear operators on $X$. Furthermore, $C(\mathbb{R},X)$ is the space of all continuous $X$-valued functions $u:\mathbb{R}\to X$
and by $C_0(\mathbb{R},X)$ we mean the Banach space of all functions in $C(\mathbb{R},X)$ vanishing at $\pm\infty$, endowed with the sup-norm
$$\|u\|_\infty=\sup\limits_{s\in\mathbb{R}} \|u(s)\|,$$
that is,
$$C_0(\mathbb{R},X)=\left\{ u\in C(\mathbb{R},X):\,\lim\limits_{|s|\to\infty} u(s)=0\right\}.$$
It is well-known that
$C_c(\mathbb{R},X)=\left\{ u\in C(\mathbb{R},X):\,supp(u) \text{ is compact}\right\}$
is dense in $C_0(\mathbb{R},X)$.

\subsection{Generalized evolution semigroups}\label{sec.3.1}

The main imperfection of the  theory of evolution semigroups  defined by \eqref{eq.ces}  is that it does not address to evolution families which are not exponentially bounded,
e.g.,
\begin{equation*}\label{eq.U3}
U(t,s)(x_1,x_2)=(e^{s^3-t^3}x_1,\,e^{t^3-s^3}x_2), \text{ for } (x_1,x_2)\in\mathbb{R}^2.
\end{equation*}
In fact, in this case we get
$$T_tu(s)=(e^{-3s^2t+3st^2-t^3}u_1(s-t),\, e^{3s^2t-3st^2+t^3}u_2(s-t)),$$
for $t\geq 0,\, u=(u_1,u_2)\in C_0(\mathbb{R},\mathbb{R}^2), \, s\in\mathbb{R}$.
One can easily observe that $\left\{  T_{t} \right\}_{t\geq 0}$ does not satisfy the semigroup inequality
$$\|T_t\|\leq Ke^{\alpha t},\, t\geq 0,$$
for some constants $\alpha>0$ and $K\geq 1$. Therefore, it is not a $C_{0}$-semigroup on $C_0(\mathbb{R},\mathbb{R}^2)$.

The aim of this section is to generalize the concept of the {evolution semigroup} associated to a not necessarily exponentially bounded evolution family, replacing the right translation semiflow in formula \eqref{eq.ces} with a certain real semiflow.

Let $\varphi:\mathbb{R}_+\times \mathbb{R}\to \mathbb{R}$ be a real semiflow and $\mathcal{U}=\left\{U(t,s)\right\}_{t\geq s}$ be an \emph{evolution family} on $X$.
By a (strongly continuous) \emph{evolution family} on $X$ we mean a collection of bounded linear operators $U(t,s)$, for $t\geq s$,  acting on $X$, such that
\begin{itemize}
\item $U(t,t)=\mathrm{Id}$, $t\in\mathbb{R}$;
\item $U(t,\tau)U(\tau,t_0)=U(t,t_0),$ $t\geq \tau\geq t_0$ in $\mathbb{R}$;
\item for each $x\in X$, the mapping $(t,s)\mapsto U(t,s)x$ is continuous on $$\Delta=\left\{(t,s)\in\mathbb{R}^2:\,t\geq s\right\}.$$
\end{itemize}

For any $t\geq 0$ and $u\in C(\mathbb{R},X)$, we define
\begin{equation}
T_{t}u(s):=U(s,\varphi_{t}(s))u(\varphi_{t}(s)), \;s\in\mathbb{R}.  \label{eq1}%
\end{equation}

One can easily observe that $T_tu\in C(\mathbb{R},X)$, which yields that the mapping $T_t:C(\mathbb{R},X)\to C(\mathbb{R},X)$ is well-defined for each $t\geq 0$.  Furthermore, we have
$$T_0=\mathrm{Id} \text{ and } T_t T_\tau=T_{t+\tau}, \text{ for }t,\tau\geq 0.$$

\begin{remark}\label{rem.degenerate}
If $\varphi$ is a degenerate semiflow, then evidently
there exists $s_{0}\in\mathbb{R}$ such that%
\begin{equation*}\label{eq.ndg}
T_{t}u\left(  s_{0}\right)  =u\left(  s_{0}\right), \text{ for all $u\in C(\mathbb{R},X)$ and $t\geq0$.}
\end{equation*}
Conversely, if the above relation holds for some $s_{0}\in\mathbb{R}$, then $\varphi$ is degenerate.
\end{remark}
Indeed, if we assume that the orbit $o(s_{0})$
is non-trivial, then there exists $t_{0}>0$ with
$$s_{1}=\varphi_{t_{0}}\left(s_{0}\right)<s_{0}.$$
Thus,
$$U\left( s_{0},s_{1}\right)  u(s_{1})=T_{t_0}u(s_0)=u(s_{0}),$$ for all mappings $u\in C(\mathbb{R},X)$.
For arbitrary $x, x_{1}, x_{2}\in X$ with $x_{1}\neq x_{2}$,  choose $u_{1}, u_{2}\in C(\mathbb{R},X)$
such that
$$u_{1}(s_{0})=x_{1},\,u_{2}(s_{0})=x_{2}, \text{  and } u_{1}(s_{1})=u_{2}(s_{1})=x.$$ This yields
$$U\left(  s_{0},s_{1} \right)x=x_{1} \text{  and }U\left(  s_{0},s_{1} \right)x=x_{2},$$
which contradicts $x_1\neq x_2$. Therefore, $o(s_0)$ is trivial and, consequently, $\varphi$ is degenerate.

\begin{remark}
If $\varphi:\mathbb{R}_+\times \mathbb{R}\to \mathbb{R}$ is a degenerate semiflow satisfying any of the following  three equivalent conditions:
\begin{enumerate}
\item[(i)] $\omega$ is bounded by above,
\item[(ii)] $\omega$ is constant in a neighborhood of $+\infty$,
\item[(iii)] there exist sequences $t_n>0$ and $s_n\to \infty$ for which $(\varphi_{t_n} (s_n) )_{n\in\mathbb{N}}$ is bounded by above,
\end{enumerate}
then
$\left\{T_{t}\right\}_{t\geq0}$, given by \eqref{eq1}, might not be a $C_{0}$-semigroup on $C_0(\mathbb{R},X)$.
\end{remark}

Let us first prove that the above conditions are equivalent if $\varphi$ is a degenerate semiflow, which means that there exists $s_0\in\mathbb{R}$ with $\omega(s_0)\in\mathbb{R}$.

(i) $\Rightarrow$ (ii) Suppose that there exists $\alpha\in\mathbb{R}$ such that $\omega(s)\leq \alpha$, for every $s\geq s_0$. Then, for $s\geq s_0$ we get that $\omega(s)\in\mathbb{R}$ and
$\omega(\omega(s))\leq \omega(\alpha)$. By \eqref{eq.prop.c}, this
yields that  $\omega(s)\leq \omega(\alpha)$, for every $s\geq s_0$. On the other hand, for  $s\geq \alpha$ we have $\omega(s)\geq \omega(\alpha)$. Hence, $\omega(s)=\omega(\alpha), \text{ for every } s\geq \max\{s_0,\alpha\},$
which proves (ii).

(ii) $\Rightarrow$ (iii) Assume that there exist $\alpha\in\mathbb{R}$ and $\tau\geq s_0$ with $\omega(s)=\alpha$ for every $s\geq \tau$. Thus, there exists $\delta>0$ such that
$$\varphi_t(s)-\alpha<1, \text{ for all }  t>\delta \text{ and } s\geq \tau.$$
For each $n\in\mathbb{N}$ choose $t_n>\delta$ and $s_n\geq \tau$ with $s_n\to\infty$. Then, $\varphi_{t_n}(s_n)<\alpha+1$ for every $n\in\mathbb{N}$, that is $(\varphi_{t_n} (s_n) )_{n\in\mathbb{N}}$ is bounded by above.

(iii) $\Rightarrow$ (i) Let $t_n> 0$ and $s_n\in\mathbb{R}$ with $s_n\to\infty$ such that
$\varphi_{t_n} (s_n) \leq \alpha$ for every $n\in\mathbb{N}$ and for some $\alpha\in\mathbb{R}$.
Fix $s\geq s_0$. Since $s_n\to\infty$, there exists $n_0\in\mathbb{N}$ such that $s_n>s$ for every $n\geq n_0.$ Then, $\omega(s_n)\geq \omega(s)$ for $n\geq n_0.$  On the other hand, from  \eqref{eq.w} we get
$\omega(s_n)\leq \alpha$ and hence $\omega(s)\leq \alpha$ for every $s\geq s_0$.

\medskip

Now, set $U(t,s)=\mathrm{Id}$, for $t\geq s$, and let $t_n\searrow 0$ and $s_n\to +\infty$ such that
$$\alpha-1<\varphi_{t_n}(s_n)<\alpha, \text{ for all } n\in\mathbb{N} \text{ and for some } \alpha\in\mathbb{R}.$$
Pick $u\in C_{c}(\mathbb{R},X)$ such that $u(t)=\xi\neq 0$ for $t\in (\alpha-1,\alpha).$ Since $s_n\to+\infty$ and $u$ has compact support, we get that $u(s_n)=0$ for large $n\in\mathbb{N}$. Then,
\begin{align*}
T_{t_n}u(s_n)-u(s_n)&=U(s_n,\varphi_{t_n}(s_n))u(\varphi_{t_n}(s_n))-u(s_n)=u(\varphi_{t_n}(s_n))=\xi\nrightarrow 0,
\end{align*}
which proves that $\left\{T_{t}\right\}_{t\geq0}$ is not a $C_{0}$-semigroup on $C_0(\mathbb{R},X)$.

\medskip

In the following result, we give conditions such that the family $\left\{T_{t}\right\}_{t\geq0}$ considered in \eqref{eq1} is a $C_{0}$-semigroup on $C_0(\mathbb{R},X)$.

\begin{theorem}\label{th1}
Assume that  $\varphi:\mathbb{R}_+\times \mathbb{R}\to \mathbb{R}$ is a non-degenerate  semiflow with
\begin{equation}\label{eq.sem}
\lim\limits_{s \to \infty} \varphi_t(s)=+\infty, \text{ for } t\geq 0,
\end{equation}
or a (degenerate) real semiflow with $\lim\limits_{s\to \infty} \omega(s)=+\infty$. If the evolution family $\mathcal{U}=\left\{U(t,s)\right\}_{t\geq s}$ satisfies inequality
\begin{equation}
\| U(s,\varphi_{t}(s))\| \leq Ke^{\alpha t},\; t\geq 0, \, s\in\mathbb{R},\label{eq7}%
\end{equation}
for some reals $\alpha>0$ and $K\geq 1$, then
$\left\{T_{t}\right\}_{t\geq0}$ defined in \eqref{eq1} is a $C_{0}$-semigroup on $C_0(\mathbb{R},X)$.
\end{theorem}

\begin{proof}
For $u\in C_0(\mathbb{R},X)$, by \eqref{eq7}, we have
$$\| T_t u(s)\|=\| U(s,\varphi_{t}(s))u(\varphi_{t}(s))\|\leq Ke^{\alpha t} \| u(\varphi_{t}(s))\|.$$

Let us first remark that if $\lim\limits_{s\to \infty} \omega(s)=+\infty$, then, by \eqref{eq.w}, we deduce that \eqref{eq.sem} holds.
Further,  \eqref{eq.-oo} and \eqref{eq.sem} imply $u(\varphi_{t}(s))\to 0$ as $s\to \pm \infty$, thus $T_t u\in C_0(\mathbb{R},X)$ and, consequently, $T_t:C_0(\mathbb{R},X)\to C_0(\mathbb{R},X)$ is well-defined for each $t\geq 0$.

Since  $C_{c}(\mathbb{R},X)$ is dense in $C_{0}(\mathbb{R},X)$, it suffices to prove that
$$\lim_{n\rightarrow\infty}\left(  T_{t_{n}}u(s_{n})-u(s_{n})  \right)=0,$$
for any $t_{n}\to 0^+$, $s_{n}\in\mathbb{R}$ and fixed $u\in C_{c}(\mathbb{R},X)$.

If $s_{n}\rightarrow-\infty$, then $\varphi_{t_{n}}(s_{n})\rightarrow-\infty$, and if $s_{n}\rightarrow +\infty$, then
$\varphi_{t_{n}}(s_{n})\rightarrow+\infty$. Since $u$ has compact support, in both situations we get
$ T_{t_{n}}u(s_{n})-u(s_{n})=0$
for sufficiently large $n\in\mathbb{N}$.

It remains to assume that the sequence $(s_{n})$ is
bounded and pick $\varepsilon>0$. Then there exists a finite set
$A_{\varepsilon}=\left\{  \eta_{1},\eta_{2}, \ldots, \eta_{k}\right\}\subset \mathbb{R}$ such that for
any $n\in\mathbb{N}$ there exists $\xi_{n}\in A_{\varepsilon}$ with
\begin{equation}
\| u( s_{n})-\xi_{n}\| <\eta, \label{eq8}%
\end{equation}
where $\eta=\frac{\varepsilon}{2(Ke^\alpha +1)}.$

Having compact support, $u$ is uniformly continuous, thus there exists $\delta_1(\varepsilon)>0$ such that
\begin{equation}\label{eq.inter1}
\| u(s^\prime)-u(s^{\prime\prime}) \|<\eta, \text{ if $|s^\prime-s^{\prime\prime}| <\delta_1(\varepsilon)$. }
\end{equation}

On the other hand, as $\Delta\ni(s,\tau)\mapsto U(s,\tau)x\in X$ is
continuous for each $x\in X$, then, setting $x=\eta_i\in A_\varepsilon$, there exists
$\delta(\varepsilon,\eta_i)>0$ such that
$|s-\tau|<\delta(\varepsilon,\eta_i)$ implies%
\begin{equation}
\| U(s,\tau)\eta_{i}-\eta_{i}\| <\eta,\,i=\overline{1,k}. \label{eq10}%
\end{equation}

Put
$$\delta(\varepsilon)=\min\left\{\delta_1(\varepsilon),\delta(\varepsilon,\eta_{1}), \ldots,\delta(\varepsilon,\eta_{k})\right\}.$$

Since $(s_{n})$ is bounded, then
$\varphi_{t_{n}}(s_{n})-s_{n}\rightarrow 0$,
thus there exists $N(\varepsilon)\in \mathbb{N}$ such that
\begin{equation}\label{eq.inter2}
|\varphi_{t_{n}}( s_{n}) -s_{n}|<\delta(\varepsilon), \text{ for } n\geq N(\varepsilon),
\end{equation}
and, by \eqref{eq.inter1}, we get
\begin{equation}
\| u(\varphi_{t_{n}}(s_{n}))-u(s_{n})\| <\eta, \text{ for } n\geq N(\varepsilon). \label{eq9}%
\end{equation}
On the other hand, \eqref{eq10} and \eqref{eq.inter2}  imply
\begin{equation}\label{eq.inter3}
\| U(s_n,\varphi_{t_n}(s_n))\xi_n-\xi_n\|<\eta, \text{ for } n\geq N(\varepsilon).
\end{equation}

Finally, using
inequalities \eqref{eq7}, \eqref{eq8}, \eqref{eq9} and \eqref{eq.inter3}, for  $n\geq N(\varepsilon)$ we have
\begin{align*}
\| T_{t_{n}}u(s_{n})-u(s_{n})\|
&\leq \| U(s_{n},\varphi_{t_{n}}(s_{n}))(u(\varphi_{t_{n}}(s_{n}))-u(s_{n})) \| \\
&\quad +\| U(s_{n},\varphi_{t_{n}}(s_{n}))(u(s_{n})-\xi_{n})\| \\
&\quad +\| U(s_{n},\varphi_{t_{n}}(s_{n}))\xi_{n}-\xi_{n}\| \\
&\quad +\| \xi_{n}-u(s_{n})\|\\
&\leq 2Ke^{\alpha t_{n}}\eta+2\eta.
\end{align*}
Therefore,
$\| T_{t_{n}}u(s_{n})-u(s_{n})\|<\varepsilon$
for sufficiently large $n$, which completes the proof.
\end{proof}

The $C_{0}$-semigroup $\left\{T_{t}\right\}_{t\geq0}$ given by Theorem \ref{th1} will be called \emph{the generalized evolution semigroup} on $C_0(\mathbb{R},X)$ associated to the real semiflow $\varphi$ and the evolution family $\mathcal{U}$.

Notice that if we put $\varphi_{t}(s)=s-t$  (in fact, the right translation semiflow),
we step over the well-known
concept of the evolution semigroup on $C_0(\mathbb{R},X)$ associated to an exponentially bounded evolution family (see \eqref{eq.ueg} and \eqref{eq.ces}).

\begin{example}
Setting
\[
\varphi_{t}\left(  s\right)  =
\begin{cases}
se^{t},& \text{ if }s<0,\\
s,& \text{ if }s\geq0,
\end{cases}
\]
one can easily check that $\varphi$ is a real semiflow with
trivial orbits (in fact, $o(s)=\left\{ s\right\}$ for all $s\geq0$) and
\[
\omega(s)=
\begin{cases}
-\infty,& \text{ if }s<0,\\
s,& \text{ if }s\geq0.
\end{cases}
\]
Furthermore,
$\lim\limits_{s\to \infty} \omega(s)=+\infty$. Then, according to Theorem \ref{th1}, the corresponding family $\left\{T_{t}\right\}_{t\geq0}$ defined by \eqref{eq1} is a $C_{0}$-semigroup on $C_0(\mathbb{R},X)$ provided that the evolution family $\mathcal{U}$ satisfies condition \eqref{eq7}. This reduces in fact to
$$\| U(s,\tau)\| \leq K e^{\alpha(\ln(-\tau)-\ln(-s))} \; \Leftrightarrow \; \| U(s,\tau)\| \leq K ({\tau}/{s})^\alpha,$$
for some constants $\alpha>0$ and $K\geq 1$, and all negative $s>\tau$. For instance, the evolution family
$$U(s,\tau)=\frac{1+|\tau|}{1+|s|}\,\mathrm{Id}, \,s\geq \tau,$$
satisfies the above inequality for $\alpha=K=1$.
\end{example}

\begin{remark}
Let $\varphi:\mathbb{R}_+\times \mathbb{R}\to \mathbb{R}$ be a continuous function such that
\begin{equation}\label{eq.phi2}
\varphi(t,s)\leq s, \text{ for all } t\geq 0 \text{ and } s\in\mathbb{R},
\end{equation}
and set
$$\varphi_t(s)=\varphi(t,s),\; s\in\mathbb{R},$$
the transition map defined by $t\geq 0$. If the family $\left\{T_{t}\right\}_{t\geq0}$ given by formula \eqref{eq1} is a
$C_{0}$-semigroup on $C_0(\mathbb{R},X)$, then $\varphi$ is a real semiflow provided that $U(t,s)\neq 0$ for all $t>s$.
\end{remark}
\begin{proof}
Indeed, it suffices to prove the relations (i) and (ii) in Definition \ref{def4}. By \eqref{eq.phi2} we have that $\varphi_0(s)\leq s$ for every $s\in\mathbb{R}$. Assume that there exists $s_0\in\mathbb{R}$ such that $\varphi_0(s_0)< s_0$. For arbitrary $x\in X$, $x\neq 0$, we choose $u\in C_c(\mathbb{R},X)$ with
$$u(\varphi_0(s_0))=x \text{ and } u(s_0)=0.$$
Since $T_0u(s_0)=u(s_0)=0$, we get
$$0=U(s_0,\varphi_0(s_0))u(\varphi_0(s_0))=U(s_0,\varphi_0(s_0))x,$$
and thus $U(s_0,\varphi_0(s_0))=0$, which is false. Then, $\varphi_0(s)=s$ for every $s\in\mathbb{R}$. To prove the second relation in Definition \ref{def4}, one can easily check that
$$T_t T_\tau=T_{t+\tau}, \text{ for }t,\tau\geq 0,$$
implies
$$U(s,\varphi_\tau(\varphi_t(s)))u(\varphi_\tau(\varphi_t(s)))=U(s,\varphi_{t+\tau}(s))u(\varphi_{t+\tau}(s)),$$
for all $t,\tau\geq 0,\,s\in\mathbb{R}, \,u\in C_0(\mathbb{R},X),$
and, proceeding as above, one may prove that $\varphi_{t}\circ\varphi_{\tau}=\varphi_{t+\tau}$, for all $t,\tau\geq 0$.
\end{proof}

Theorem \ref{th4} shows that
if $\varphi$
is a non-degenerate semiflow, then there exists $\mu\in\mathcal{A}$ such that
\begin{equation}
T_{t}u(s)=U(s,\mu^{-1}(\mu(s)-t))  u(\mu^{-1}(\mu(s)-t)),  \label{eq22}%
\end{equation}
for all $t\geq 0$, $u\in C(\mathbb{R},X)$, and $s\in\mathbb{R}$. On the other hand,  \eqref{eq.sem} is equivalent to
\begin{equation}\label{eq.loo}
\ell=\lim\limits_{s\rightarrow+\infty}\mu (s)=+\infty,
\end{equation}
and we will denote $\mathcal{A}_{\,\infty}$ the set of all $\mu\in\mathcal{A}$ satisfying \eqref{eq.loo}. Furthermore,
the inequality \eqref{eq7} becomes
\begin{equation}
\| U(s,\tau)\| \leq K e^{\alpha (\mu(s) -\mu(\tau))}, \text{ for } s\geq \tau. \label{eq18}%
\end{equation}

Therefore, we deduce a particular class of generalized evolution semigroups which is emphasized in the next result.

\begin{corollary}\label{th.ges-nd}
Assume that
$\mu\in\mathcal{A}_{\, \infty}$ and the evolution family $\mathcal{U}$ satisfies \eqref{eq18} for some constants
$\alpha>0$ and $K\geq 1$. Then
the family $\left\{T_{t}\right\}_{t\geq0}$ given by \eqref{eq22} is a $C_{0}$-semigroup on $C_0(\mathbb{R},X)$.
\end{corollary}

If $\rho:\mathbb{R}\to [0,\infty)$ is a continuous function such that the set $\{t\in\mathbb{R}: \rho(t)>0\}$ is dense in $\mathbb{R}$,
\begin{equation}\label{cond.ro}
\lim\limits_{t\to \infty} \int_0^t \rho(\xi) \, d\xi =+\infty \text{ and } \lim\limits_{t\to -\infty} \int_t^0 \rho(\xi) \, d\xi =+\infty,
\end{equation}
(see \cite[Definition 4 (iii)]{Martin} or  \cite[Definition 2.3]{Jiang}) and there exist constants $\alpha>0$ and $K\geq 1$ with
\begin{equation}\label{eq.gen}
\| U(s,\tau)\| \leq K e^{\alpha\int_\tau^s \rho(\xi)\,d\xi}, \; s\geq \tau,
\end{equation}
then the function defined by $$\mu(t)=\int_0^t \rho(\xi) \, d\xi, \, t\in\mathbb{R},$$ belongs to $\mathcal{A}_{\, \infty}$, it is continuously differentiable  and \eqref{eq18} holds.
Inequality  \eqref{eq.gen} is in fact the hypothesis of $\rho$-bounded growth assumed in \cite{Mul}.
Notice that any evolution family generated by a differential equation $\dot{x%
}=A(t)x$, where $\mathbb{R}\ni t\mapsto A(t)\in \mathcal{B}(X)$ is
continuous in uniform operator topology, satisfies \eqref{eq.gen}
for $\rho(t)=\| A(t) \|$.

\subsection{Similar semigroups}\label{p.similar}

In the following, let the assumptions of Corollary \ref{th.ges-nd} hold, that is
$\mu:\mathbb{R}\to \mathbb{R}$ is a continuous, strictly increasing function such that
$$\lim\limits_{s\rightarrow-\infty}\mu (s) =-\infty \text{ and } \lim\limits_{s\rightarrow+\infty}\mu (s) =+\infty,$$
and the evolution family $\mathcal{U}$ satisfies \eqref{eq18}.

Set
$$V(t,s)=U(\mu^{-1}(t),\mu^{-1}(s)), \text{ for } t\geq s.$$
One can easily check that $\mathcal{V}=\left\{ V(t,s)\right\}_{t\geq s}$ is an evolution family on $X$ satisfying
\begin{equation*}
\| V(t,s) \| \leq K e^{\alpha (t-s)}, \text{ for } t\geq s,
\end{equation*}
and thus $\mathcal{V}$ is exponentially bounded.

From the classical theory of evolution semigroups, we are able now to define
the evolution semigroup on $C_0(\mathbb{R},X)$ associated to the evolution family $\mathcal{V}$, that is
a $C_{0}$-semigroup on $C_0(\mathbb{R},X)$, given by
\begin{equation*}
S_tv(s)=V(s,s-t)v(s-t),\, t\geq 0, v\in C_0(\mathbb{R},X), s\in\mathbb{R},
\end{equation*}
and denote $(A,D(A))$ its generator.

We have
$$S_tv(s)=U(\mu^{-1}(s),\mu^{-1}(s-t))v(s-t)$$
and so
$$S_tv(\mu(s))=U(s,\mu^{-1}(\mu(s)-t))v(\mu(s)-t).$$
Letting $v=u\circ \mu^{-1}\in C_0(\mathbb{R},X)$ for $u\in C_0(\mathbb{R},X)$, we get
$$(S_t(u\circ \mu^{-1}))(\mu(s))=U(s,\mu^{-1}(\mu(s)-t))u(\mu^{-1}(\mu(s)-t))=T_tu(s).$$
Hence,
\begin{equation}\label{eq.sim1}
(S_t(u\circ \mu^{-1})) \circ \mu=T_t u, \text{ for every } u\in C_0(\mathbb{R},X).
\end{equation}
Define now the operator
\begin{equation*}\label{eq.F}
\mathcal{F}: C_0(\mathbb{R},X) \to C_0(\mathbb{R},X), \; \mathcal{F}u=u\circ \mu^{-1}.
\end{equation*}
Obviously, $\mathcal{F}$ is an invertible, bounded linear operator on $C_0(\mathbb{R},X)$, with its inverse
\begin{equation*}\label{eq.F-1}
\mathcal{F}^{-1}v=v\circ \mu.
\end{equation*}
From \eqref{eq.sim1} we have
$$T_t u= (S_t(\mathcal{F}u)) \circ \mu=\mathcal{F}^{-1}(S_t(\mathcal{F} u)), \text{ for every } u\in C_0(\mathbb{R},X),$$
and thus
\begin{equation*}\label{eq.similar}
T_t=\mathcal{F}^{-1} S_t\, \mathcal{F}.
\end{equation*}
The above relation shows that \emph{$\left\{S_{t}\right\}_{t\geq0}$ and $\left\{T_{t}\right\}_{t\geq0}$ are similar semigroups on $C_0(\mathbb{R},X)$} (see   \cite[p. 59]{Eng}).
Hence, the generator $G$ of the generalized evolution semigroup $\left\{T_{t}\right\}_{t\geq0}$ is given by
\begin{equation}\label{DAG}
Gu=(\mathcal{F}^{-1} A\, \mathcal{F})u=A\,(u\circ \mu^{-1})\circ \mu,
\end{equation}
with domain
\begin{equation}\label{DAG-1}
D(G)=\left\{ u\in C_0(\mathbb{R},X):\, u\circ \mu^{-1}\in D(A) \right\}.
\end{equation}
Furthermore, $\sigma(G)=\sigma(A)$.

This result is of a significant importance because one can now simply deduce the spectral mapping theorem for the generalized evolution semigroup, given by Corollary \ref{th.ges-nd}, from the classical theory of evolution semigroups.

\begin{corollary}\label{th2}
Under the assumptions of Corollary \ref{th.ges-nd}, let $G$ be the generator of the generalized evolution semigroup $\left\{T_{t}\right\}_{t\geq0}$.
The spectrum $\sigma(T_t)$ is rotationally invariant for  $t>0$ and the spectrum $\sigma(G)$ is invariant under translations along the imaginary axis. Furthermore,
the generalized evolution semigroup satisfies the spectral mapping theorem
\begin{equation*}\label{smt}
\sigma\left(  T_{t} \right)\setminus\{ 0 \}  =e^{t\sigma(G)}, \text{ for } t\geq 0,
\end{equation*}
and, consequently, the following statements are equivalent:
\begin{enumerate}
\item $\sigma(T_t)\cap \mathbb{T}=\emptyset$, for some $t>0$,  and thus for all $t>0$;
\item $0\in \varrho(G)$, that is $G:D(G)\subseteq C_0(\mathbb{R},X) \to C_0(\mathbb{R},X)$ is invertible, where $\varrho(G)$ is the resolvent of $G$.
\end{enumerate}
\end{corollary}

In the following result we connect the generator of the generalized evolution semigroup to the solvability of an integral equation (see \cite[Lemma 1]{Min3} or \cite[Lemma 1.1]{Min1}).

\begin{proposition}\label{lem.int}
Let $u,f\in C_0(\mathbb{R},X)$. Then $u\in D(G)$ and $Gu=-f$ if and only if
\begin{equation*}\label{eq.int1}
u(t)=U(t,s)u(s)+\int_s^t \mu^\prime(\xi) U(t,\xi)f(\xi)\,d\xi, \text{ for } t\geq s.
\end{equation*}
\end{proposition}

\begin{proof}
Let $u\in D(G)$ and $f\in C_{0}(\mathbb{R},X)$
such that $Gu=-f$. By \eqref{DAG}--\eqref{DAG-1} we get
\begin{equation*}
v=u\circ \mu ^{-1}\in D(A)\text{ and }Av=-f\circ \mu ^{-1}.
\end{equation*}
Lemma 1 in \cite{Min3} implies
\begin{equation*}
v(t)=V(t,s)v(s)+\int_{s}^{t}V(t,\tau )f(\mu ^{-1}(\tau ))d\tau ,\;t\geq
s,
\end{equation*}%
which is equivalent to
\begin{equation*}
u(\mu ^{-1}(t))=U(\mu ^{-1}(t),\mu ^{-1}(s))u(\mu
^{-1}(s))+\int_{s}^{t}U(\mu ^{-1}(t),\mu ^{-1}(\tau ))f(\mu
^{-1}(\tau ))d\tau.
\end{equation*}%
Replacing  $t$ and $s$ by $\mu
(t)$ and $\mu (s)$, respectively, we get
\begin{align*}
u(t)& =U(t,s)u(s)+\int_{\mu (s)}^{\mu (t)}U(t,\mu ^{-1}(\tau
))f(\mu ^{-1}(\tau ))d\tau  \\
& =U(t,s)u(s)+\int_{s}^{t}\mu ^{\prime }(\xi )U(t,\xi )f(\xi )d\xi ,
\end{align*}%
which proves the desired formula. The converse can be proved by
reversing all the above arguments.
\end{proof}

\section{Applications of generalized evolution semigroups to general dichotomies}\label{sec.4}

Using the framework constructed in the previous sections, we are able now to completely characterize a wide class of dichotomies of an evolution family  via its corresponding generalized evolution semigroup.

We recall that a $C_0$-semigroup $\left\{T_{t}\right\}_{t\geq0}$ on a Banach space $X$ is called \emph{hyperbolic} if
$$\sigma(T_t)\cap \mathbb{T} =\emptyset \text{ for some/all } t>0,$$
where $\mathbb{T}=\{\lambda\in\mathbb{C}:\,|\lambda|=1\}$. It is well-known that the $C_0$-semigroup $\left\{T_{t}\right\}_{t\geq0}$ is hyperbolic if and only if there exists a projection $\mathcal{P}$ on $X$ satisfying
$$T_t \mathcal{P}=\mathcal{P} T_t,\; t\geq 0,$$
and the following conditions hold:
\begin{itemize}
\item the map $T_t|_{\mathcal{Q}X}: \mathcal{Q}X \to \mathcal{Q}X$ is invertible for each $t\geq 0$, where $\mathcal{Q}=\mathrm{Id}-\mathcal{P}$;
\item there exists $\nu>0$ and $N\geq 1$ such that
\begin{equation*}\label{hyp1}
\| T_t \mathcal{P} x \|\leq N e^{-\nu t} \|x\| \text{ and } \| \left(T_t^{\mathcal{Q}}\right)^{-1} \mathcal{Q} x \|\leq N e^{-\nu t} \|x\|,
\end{equation*}
for $t\geq 0$ and $x\in X$, where $\left(T_t^{\mathcal{Q}}\right)^{-1}$ is the inverse of $T_t|_{\mathcal{Q}X}: \mathcal{Q}X \to \mathcal{Q}X$.
\end{itemize}

In fact, the structural projection $\mathcal{P}$ is the Riesz projection corresponding to the operator $T_{t_0}$ and  the spectral set  $\sigma(T_{t_0})\cap \mathbb{D}$,
$$\mathcal{P}=\frac{1}{2\pi i} \int_{\mathbb{T}}(\lambda - T_{t_0})^{-1} d\lambda,$$
for some fixed $t_0>0$ (without loss of generality, one may consider $t_0=1$).

In the following, assume that $\mu\in\mathcal{A}_{\, \infty}$ and $\mathcal{U}=\left\{U(t,s)\right\}_{t\geq s}$ is an evolution family on $X$ satisfying \eqref{eq18}.
Let
$\left\{  T_{t}\right\}_{t\geq0}$ be the corresponding generalized evolution semigroup given by Corollary \ref{th.ges-nd} and denote $(G,D(G))$ its generator.

\begin{proposition}\label{lemma6}
If the generalized evolution semigroup $\left\{  T_{t}\right\}_{t\geq0}$ is hyperbolic, having structural projection $\mathcal{P}$, then
$$\mathcal{P}(\beta u)=\beta\mathcal{P}u,$$
for all $\beta\in C_{b}(\mathbb{R})$ and  $u\in C_0(\mathbb{R},X)$, where $C_{b}(\mathbb{R})$ denotes the space of all bounded continuous real-valued functions.
\end{proposition}

\begin{proof}
We first remark that
$$\mathrm{Range}(\mathcal{P})=\left\{  u\in C_0(\mathbb{R},X):\; \lim\limits_{t\to\infty}T_{t}u=0\right\}.$$
Therefore, if $\beta\in C_{b}(\mathbb{R})$ and  $u\in C_0(\mathbb{R},X)$, then $\|T_t \mathcal{P} u\|_\infty \to 0$ as $t\to \infty$ and
$$\left\Vert T_{t}(\beta\mathcal{P}u)(s)\right\Vert=| \beta(\mu^{-1}(\mu(s)-t))|\, \|T_t \mathcal{P} u (s)\| \leq \|\beta\|_\infty\, \|T_t \mathcal{P} u \|_\infty.$$
This yields that $\beta \mathcal{P}u\in \mathrm{Range}(\mathcal{P})$ and $\mathcal{P}(\beta \mathcal{P}u)=\beta \mathcal{P}u.$
On the other hand, we have
\begin{align*}
\| \mathcal{P}(\beta \mathcal{Q}u)\|_{\infty}
&=\left\| \mathcal{P} \left(\beta T_t \left(T_t^{\mathcal{Q}}\right)^{-1} \mathcal{Q}u \right)\right\|_{\infty}\\
&=\left\| \mathcal{P} \left( T_t \beta_\mu \left(T_t^{\mathcal{Q}}\right)^{-1} \mathcal{Q}u \right)\right\|_{\infty}\\
&=\left\|  T_t \mathcal{P} \left(\beta_\mu \left(T_t^{\mathcal{Q}}\right)^{-1} \mathcal{Q}u \right)\right\|_{\infty},
\end{align*}
for some function $\beta_\mu\in C_{b}(\mathbb{R})$, which depends on $\beta$ and $\mu$, with $\|\beta_\mu\|_\infty\leq \|\beta\|_\infty$.
Let us estimate
$$
\| \mathcal{P}(\beta \mathcal{Q}u)\|_{\infty}\leq N e^{-\nu t} \| \beta \|_\infty \, \| \left(T_t^{\mathcal{Q}}\right)^{-1} \mathcal{Q}u \|_\infty
\leq N^2 e^{-2\nu t} \| \beta \|_\infty\, \|u \|_\infty\to 0,$$
as $t\to\infty$.
Thus, $\mathcal{P}(\beta \mathcal{Q}u)=0$ and, consequently,
$$\mathcal{P}(\beta u)=\mathcal{P}\left(\beta\mathcal{P}u\right)+\mathcal{P}\left(\beta\mathcal{Q}u\right)=\beta\mathcal{P}u.$$
\end{proof}

An important consequence of the above result is that
\begin{equation}\label{eq.proj}
\mathcal{P}u(s)=P(s)u(s), \;  u\in C_0(\mathbb{R},X), \,s\in\mathbb{R},
\end{equation}
for some bounded, strongly continuous, projection-valued function $P:\mathbb{R}\to \mathcal{B}(X)$ (see \cite[Proposition 9.13]{Eng}).

The following definition generalizes the classical notion of (uniform) exponential dichotomy of an evolution family.

\begin{definition}\rm\label{def.ged}
The evolution family $\mathcal{U}$  is said to admits a  \emph{$\mu$-exponential dichotomy} if there exist projections $P(t):X\to X$, $t\in \mathbb{R}$, compatible with $\mathcal{U}$ and there exist constants $\nu>0$, $N\geq 1$ such that
\begin{equation*}\label{eq.e.d.mu}
\| U(t,s)P(s)\| \leq Ne^{-\nu (\mu(t)-\mu(s)) }  \text{ and }  \| U_Q(s,t)Q(t)\| \leq Ne^{-\nu (\mu(t)-\mu(s))},
\end{equation*}
for all $t\geq s$.

Notice that  $\|P(t)\|\leq N$ for every $t\in\mathbb{R}$. Furthermore, as in \cite[Lemma 4.2]{Min1}, one may prove that the mapping $t\mapsto P(t)$ is strongly continuous and thus $P(\cdot)\in C_b(\mathbb{R},\mathcal{B}_s(X))$, the space of all bounded and continuous functions from $\mathbb{R}$ with values in $\mathcal{B}(X)$ endowed with the topology of strong convergence.

If $\mu$ is continuously differentiable with $\rho(s)=\mu^\prime(s)$,  the above inequalities are replaced by
\begin{equation*}\label{eq.e.d.ro}
\| U(t,s)P(s)\| \leq N e^{-\nu\int_s^t \rho(\xi)\,d\xi}  \text{ and }  \| U_Q(s,t)Q(t)\| \leq Ne^{-\nu\int_s^t \rho(\xi)\,d\xi},
\end{equation*}
for $t\geq s$, and thus we extend the concept of a generalized exponential dichotomy (firstly introduced, to our knowledge, by Martin in \cite[Definition 4]{Martin}) to arbitrary evolution families (see also \cite{Jiang,Jiang2,Mul}). In particular, if $\mu(s)=s$, we step over the usual exponential behavior. On the other hand, letting
\begin{equation}\label{polynom}
\mu(s)=\mathrm{sign}(s)\ln(1+|s|),\, s\in\mathbb{R},
\end{equation}
we get the polynomial behavior on the real line.
\end{definition}

The following question is natural and leads to the main goal of this section:
what is the connection between the existence of  $\mu$-exponential dichotomy of an evolution family and the hyperbolicity of the corresponding generalized evolution semigroup?

Before giving an answer to this question, we consider a simple example.

\begin{example}
Let $\mu$ be the polynomial growth rate defined by \eqref{polynom}. Obviously, the evolution family
$$U(t,s)(x_1,x_2)=(e^{\mu(s)-\mu(t)}x_1,\,e^{\mu(t)-\mu(s)}x_2), \text{ for } t\geq s \text{ and } (x_1,x_2)\in\mathbb{R}^2,$$
admits a $\mu$-exponential dichotomy.
On the other hand,
the generalized evolution semigroup corresponding to $\mathcal{U}$, given by
$$T_tu(s)=(e^{-t}u_1(\mu^{-1}(\mu(s)-t)),\,e^t u_2(\mu^{-1}(\mu(s)-t))),$$
for $t\geq 0$, $u=(u_1,u_2)\in C_0(\mathbb{R}, \mathbb{R}^2)$,  $s\in\mathbb{R}$, is hyperbolic.
Furthermore, one may show that the
generator $G$ of $\left\{  T_{t}\right\}_{t\geq0}$ is
$$G(u_1,u_2)(s)=\left(-u_1(s)-\frac{1}{\mu^\prime(s)} u_1^\prime(s),u_2(s)-\frac{1}{\mu^\prime(s)} u_2^\prime(s)\right),   \text{ $s\in\mathbb{R}$},$$
with its domain
$D(G)=\{u\in C_0(\mathbb{R}, \mathbb{R}^2)\cap C^1(\mathbb{R}, \mathbb{R}^2):\, \lim\limits_{s\to\pm\infty} u^\prime(s)=0\}$.
\end{example}

The following  result   characterizes the concept of $\mu$-exponential dichotomy of an evolution family in terms of spectral properties of the generator of the corresponding generalized evolution semigroup.

\begin{theorem}[Dichotomy Theorem]\label{th3}
The following statements are equivalent:
\begin{enumerate}
\item[(a)] $\mathcal{U}$ admits a $\mu$-exponential dichotomy;
\item[(b)] $\left\{T_{t}\right\}_{t\geq0}$ is hyperbolic;
\item[(c)] $\sigma(G)\cap i\mathbb{R}=\emptyset;$
\item[(d)] $G$ is invertible and, in particular, if $\mu$ is continuously differentiable, then its inverse is given by the formula
\begin{equation}\label{eq.G.fomula}
(G^{-1}f)(t)=-\int_\mathbb{R} \mu^\prime (\xi) \Gamma(t,\xi)f(\xi)\, d\xi,
\end{equation}
for all  $f\in C_0(\mathbb{R}, X) \text{ and } t\in\mathbb{R}$, where $\Gamma$ is the Green function associated to $\mathcal{U}$.
\end{enumerate}
\end{theorem}
\begin{proof}
We emphasize that  $\mathcal{U}$ admits a $\mu$-exponential dichotomy with respect to  projections $P(t)$ if and only if the evolution family $\mathcal{V}$ defined in Paragraph \ref{p.similar} admits a (uniform) exponential dichotomy with respect to
$$P_{\mu}(t)=P(\mu^{-1}(t)), \; t\in\mathbb{R}.$$
Therefore, the equivalences (a) $\Leftrightarrow$ (b) $\Leftrightarrow$ (c) hold from
Paragraph \ref{p.similar} and Theorem \ref{th.ued}. It remains to prove formula \eqref{eq.G.fomula}.
For this pick $f\in C_0(\mathbb{R}, X)$ and set
\begin{equation*}\label{eq.Green}
u(t)=-\int_\mathbb{R} \mu^\prime (\xi) \Gamma(t,\xi)f(\xi)\, d\xi,\;t\in\mathbb{R}.
\end{equation*}
By Theorem \ref{th.ued} we get
\begin{align*}
A^{-1}(f\circ \mu^{-1})(\mu(t))&=
-\int_{-\infty}^{\mu(t)} V(\mu(t),\tau)P_{\mu}(\tau)f(\mu^{-1}(\tau))\,d\tau\\
&\qquad +\int_{\mu(t)}^{+\infty}  V_{Q_\mu}(\mu(t),\tau)Q_\mu(\tau)f(\mu^{-1}(\tau)) \,d\tau\\
&=-\int_{-\infty}^{\mu(t)} U(t,\mu^{-1}(\tau))P(\mu^{-1}(\tau))f(\mu^{-1}(\tau))\,d\tau\\
&\qquad +\int_{\mu(t)}^{+\infty}  U_{Q}(t,\mu^{-1}(\tau))Q(\mu^{-1}(\tau))f(\mu^{-1}(\tau)) \,d\tau\\
&=-\int_{-\infty}^{t}\mu^\prime(\xi) U(t,\xi)P(\xi)f(\xi)\,d\xi\\
&\qquad +\int_{t}^{+\infty} \mu^\prime(\xi) U_{Q}(t,\xi)Q(\xi)f(\xi) \,d\xi\\
&=u(t), \, t\in\mathbb{R}.
\end{align*}
This yields
$$A^{-1}(f\circ \mu^{-1})\circ \mu=u.$$
Thus, $u\circ \mu^{-1}\in D(A)$ and, equivalently, $u\in D(G)$. On the other hand, by  \eqref{DAG} we have
$G^{-1}f=A^{-1}(f\circ \mu^{-1})\circ \mu$, which proves the desired formula.
\end{proof}

We emphasize that formula \eqref{eq.proj} provides the relation between the projection $\mathcal{P}$, that corresponds to the hyperbolic generalized evolution semigroup, and the dichotomy projections $P(t)$  given by  Definition \ref{def.ged}.

\begin{example}
For any fixed continuously differentiable function $\mu\in\mathcal{A}_{\, \infty}$  and for any projection $P\in\mathcal{B}(X)$, we consider the evolution family
$$U(t,s)=e^{\mu(s)-\mu(t)}P+e^{\mu(t)-\mu(s)}Q,$$
where $Q=\mathrm{Id}-P$.

Even if one can easily observe that $\mathcal{U}$ admits a $\mu$-exponential dichotomy, the purpose of this example is to show that the generator $G$ corresponding to the generalized evolution semigroup associated to $\mathcal{U}$ is invertible, without explicitly constructing it. Pick $f\in C_0(\mathbb{R}, X)$, $f\neq 0$, and set
\begin{equation}\label{eq.Green-ex}
u(t)=\int_{-\infty}^t \mu^\prime (\xi) U(t,\xi)Pf(\xi)\, d\xi-\int_t^{\infty} \mu^\prime (\xi) U(t,\xi)Qf(\xi)\, d\xi,\;t\in\mathbb{R}.
\end{equation}
One may  check that the integrals in \eqref{eq.Green-ex} are convergent, thus $u$ is well-defined.  Let $\varepsilon >0$.  Since $\lim\limits_{t\to \pm \infty} f(t)=0$, there exists $\delta_1>0$ such that
$$\| f(t) \| <\frac{\varepsilon}{4}, \text{ for every }t\in(-\infty,-\delta_1)\cup (\delta_1,+\infty).$$
On the other hand, from the convergence of the integral $\int_0^{\infty} e^{-\xi} d\xi$, there exists $\delta_2>0$ such that
$$ \int_{t'}^{t''} e^{-\xi} d\xi<\frac{\varepsilon}{4  \|f\|_{\infty}}, \text{ for every } t', t''>\delta_2.$$
Set $\delta=\max\{ \delta_1,\delta_2\}$ and choose $t>\mu^{-1}(\mu(\delta)+\delta)$, which implies $t>\delta$.
Now, we successively have
\begin{align*}
\| u(t) \|
&\leq  \int_{-\infty}^{t} \mu^\prime (\xi) e^{-(\mu(t)-\mu(\xi))}\| f(\xi) \|\, d\xi+  \int_{t}^{\infty} \mu^\prime (\xi) e^{-(\mu(\xi)-\mu(t))}\| f(\xi) \|\, d\xi\\
&\leq  \int_{-\infty}^{-\delta} \mu^\prime (\xi) e^{-(\mu(t)-\mu(\xi))}\| f(\xi) \|\, d\xi+ \int_{-\delta}^{\delta} \mu^\prime (\xi) e^{-(\mu(t)-\mu(\xi))}\| f(\xi) \|\, d\xi\\
&\qquad+\int_{\delta}^t \mu^\prime (\xi) e^{-(\mu(t)-\mu(\xi))}\| f(\xi) \|\, d\xi+\frac{\varepsilon}{4} \int_{t}^{\infty} \mu^\prime (\xi) e^{-(\mu(\xi)-\mu(t))}\, d\xi\\
&\leq \frac{\varepsilon}{4} \int_{-\infty}^{-\delta} \mu^\prime (\xi) e^{-(\mu(t)-\mu(\xi))}\, d\xi+ \int_{-\delta}^\delta \mu^\prime (\xi) e^{-(\mu(t)-\mu(\xi))}\, d\xi\,\|f\|_{\infty}\\
&\qquad + \frac{\varepsilon}{4} \int_{\delta}^t \mu^\prime (\xi) e^{-(\mu(t)-\mu(\xi))}\, d\xi+\frac{\varepsilon}{4}\\
&\leq\frac{\varepsilon}{2} \int_{-\infty}^{t} \mu^\prime (\xi) e^{-(\mu(t)-\mu(\xi))}\, d\xi+\int_{\mu(t)-\mu(\delta)}^{\mu(t)-\mu(-\delta)} e^{-\tau}\, d\tau \, \|f\|_{\infty}+\frac{\varepsilon}{4}<\varepsilon,
\end{align*}
which shows that  $\lim\limits_{t\to +\infty} u(t)=0$. Similarly, one may prove that $\lim\limits_{t\to -\infty} u(t)=0$ and thus $u\in C_0(\mathbb{R}, X)$.
Furthermore,
for $t\geq s$ we get
$$u(t)-U(t,s)u(s)=\int_s^t \mu^\prime (\xi) U(t,\xi)f(\xi)\,d\xi.$$
From Proposition \ref{lem.int}, this yields that $u\in D(G)$ and $Gu=-f$. Therefore, $G$ is invertible and the evolution family  $\mathcal{U}$ admits a $\mu$-exponential dichotomy.
\end{example}

\section*{Comments}
Throughout this paper, all the results we obtain stand for  $C_0(\mathbb{R},X)$. Evidently they can be extended  similarly to the general context of nonuniform behavior, by formally replacing the function space $C_0(\mathbb{R},X)$ with the super-space $C_*$ introduced in \cite{BPV1}. However we think that such an approach would significantly complicate computations, without adding any essential merit for our main purposes.

On the other hand, we decided to develop our considerations in the case of  evolution families on the whole real line. We believe that the half-line situation can also be treated similarly, with only a few modification (e.g., \cite{Min1}), but this is not our purpose.

\section*{Open problems}
Authors address the following open questions:
What happens if the mappings $s\mapsto\varphi_{t}(s)$, $t>0$,
are bounded, i.e. $\ell\in\mathbb{R}$?
Can we define in this case an evolution semigroup on some appropriate Banach function space? If the answer is positive, how can we apply it in the study of asymptotic behavior of the underlying evolution family?


\end{document}